\documentclass[a4paper,12pt]{amsart}

\usepackage{a4wide}
\usepackage{pgf,tikz}
\usepackage{amssymb}
\usepackage{enumerate}

\usepackage{amsmath}

\newif\ifdetails
\detailstrue
\newcommand{\DETAIL}[1]%
{\ifdetails\par\fbox{\begin{minipage}{0.9\linewidth}\textit{Detail:}
      #1\end{minipage}}\par\fi}
\newcommand{\TODO}[1]%
{\ifdetails\par\fbox{\begin{minipage}{0.9\linewidth}\textbf{TODO:}
      #1\end{minipage}}\par\fi}

\usepackage{makecell}
\usepackage{relsize}

\newtheorem{lemma}{Lemma}

\newtheorem{theorem}[lemma]{Theorem}
\newtheorem{corollary}[lemma]{Corollary}
\theoremstyle{remark}

\newtheorem{remark}{Remark}

\newtheorem{definition}[lemma]{Definition}

\DeclareMathOperator{\N}{N}
\DeclareMathOperator{\E}{ecc}

\usepackage{caption}
\usepackage{subcaption}

\usepackage{float} 
\usepackage[pdftex,a4paper,
citecolor = blue, colorlinks=true,urlcolor=blue]{hyperref}
\usepackage{mathrsfs}
\usetikzlibrary{arrows}

\newcommand{\old}[1]{{}}

\title{Wiener index, number of subtrees, and tree eccentric sequence}

\author{Peter Dankelmann}
\author{Audace A. V. Dossou-Olory}

\thanks{The work is supported by the National Research Foundation (NRF) of South Africa, grant number 118521}

\address{Peter Dankelmann \and Audace A. V. Dossou-Olory \\ Department of Mathematics and Applied Mathematics \\ University of Johannesburg \\ P.O. Box 524, Auckland Park, Johannesburg 2006, South Africa}

\email{audace@aims.ac.za}

\email{pdankelmann@uj.ac.za}


\begin{document}

\begin{abstract}
The eccentricity of a vertex $u$ in a connected graph $G$ is the distance between $u$ and a vertex farthest from it; the eccentric sequence of $G$ is the nondecreasing sequence of the eccentricities of $G$. In this paper, we determine the unique tree that minimises the Wiener index, i.e. the sum of distances between all unordered vertex pairs, among all trees with a given eccentric sequence. We show that the same tree maximises the number of subtrees among all trees with a given eccentric sequence, thus providing another example of negative correlation between the number of subtrees and the Wiener index of trees. Furthermore, we provide formulas for the corresponding extreme values of these two invariants in terms of the eccentric sequence. As a corollary to our results, we determine the unique tree that minimises the edge Wiener index, the vertex-edge Wiener index, the Schulz index (or degree distance), and the Gutman index among all trees with a given eccentric sequence. 
\end{abstract}

\maketitle

\section{Introduction}
The eccentricity $\text{ecc}_G(u)$ of a vertex $u$ in a connected graph $G$ is defined as the distance between $u$ and a vertex farthest from it, that is $$\text{ecc}_G(u)=\max_{v \in V(G)} d_G (u, v)\,,$$ where $d_G (u, v)$ denotes the distance between $u$ and $v$ in $G$, i.e. the length of a shortest $u-v$ path in $G$. The \emph{eccentric sequence} of $G$ is defined as the nondecreasing sequence of the eccentricities of $G$. It is the second oldest sequence associated with a graph, after the degree sequence. A sequence of positive integers is said to be eccentric if it is the eccentric sequence of some graph. The study of eccentric sequences in graphs was initiated in a 1975 paper~\cite{Lesniak1975} by Lesniak who showed that each entry, except possibly the smallest in an eccentric sequence, appears at least twice. In~\cite{Behzad1976} Behzad and Simpson gave a further necessary condition for a sequence to be eccentric, and also found few properties of graphs with a given eccentric sequence. Deciding if a given sequence of integers is eccentric is, in general, difficult~\cite[Problem~1]{Buckley2002}. An eccentric sequence $S$ with $m$ distinct entries is called \emph{minimal} if it has no proper eccentric subsequence with $m$ distinct entries. Lesniak~\cite{Lesniak1975} showed that $S$ is eccentric if and only if it has a subsequence with $m$ distinct entries which is eccentric. Unfortunately, deciding if a given sequence of integers is a minimal eccentric sequence appears to be difficult~\cite{Haviar1997,Monoszova2005,Nandakumar1986}. Considering a restriction to graph classes, Dankelmann et al.~\cite{Dankelmann2014} found a characterisation of eccentric sequences of maximal outerplanar graphs; Lesniak~\cite{Lesniak1975} provided a complete characterisation of tree eccentric sequences. A sequence of integers is tree eccentric if it is the eccentric sequence of some tree. A tree is said to be a \emph{caterpillar} if removing all \emph{pendant} vertices, i.e. vertices of degree $1$, produces a path. It was shown implicitly in~\cite{Lesniak1975} that every tree eccentric sequence is the eccentric sequence of some caterpillar. The same observation was later made explicit by Skurnik~\cite{Skurnik1999}, who also determined the exact number of nonisomorphic caterpillars with a given eccentric sequence.

\medskip
This paper is concerned with two problems: \emph{Determine both an exact sharp lower bound on the Wiener index of trees and an exact sharp upper bound on the number of subtrees of trees with a prescribed eccentric sequence.} The \emph{Wiener index} of a graph $G$ is defined as the sum of distances between all unordered pairs of vertices of $G$, while the \emph{number of subtrees} of $G$ is the number of subgraphs of $G$ which are trees.

\medskip
The Wiener index was introduced in 1947 by the chemist H. Wiener~\cite{Wiener1947} who observed its correlation with the physical, chemical and biological properties of certain molecules and molecular compounds. Besides its chemical applications, the Wiener index is also of great interest in graph theory~\cite{plesnik1984sum}. Moreover, research has shown a `negative' correlation between the Wiener index and other distance based topological indices~\cite{wagner2007correlation}. The minimum and maximum Wiener indices of a connected graph in terms of order (number of vertices) are attained by the complete graph and the path, respectively. By placing further restrictions on graphs, one obtains interesting subclasses: see~\cite{DobryninAll} for a survey on extremal results for the Wiener index of trees. In particular, the maximum and minimum Wiener indices of a $n$-vertex tree are attained by the path and the star, respectively. Cambie~\cite{Cambie2018} obtained an asymptotically (as $n \to \infty$) sharp upper bound for the Wiener index of a graph with order $n$ and \emph{diameter} (the maximum eccentricity) at most $d>2$, and also proved a somewhat analogous result for trees. There are other similar results for the Wiener index which prescribe constraints such as minimum degree, edge-connectivity, vertex-connectivity, independence number~\cite{DankelmannMinDegree,DankelmannEtringer,DankelmannMukwembiEdge,DankelmannMukwembi,MukwembiIndepend}; see also the survey~\cite{xu2014survey} for more information. Considering trees with a prescribed degree sequence, the so-called
greedy trees minimise the Wiener index~\cite{WangP2009,wang2008extremal,zhang2008wiener}, while the maximisation problem can be reduced to the study of caterpillars~\cite{Cela,Shi1993}. The recent paper~\cite{LinSong} studied trees that minimise the Wiener index in the class of all trees with a given segment sequence. In~\cite{PrevPaper} the authors showed that the problem of maximising the Wiener index with a given segment sequence leads to the study of the so-called quasi-caterpillars.

\medskip
A subtree of a tree $T$ is a connected subgraph of $T$. The parameter \emph{number of subtrees} of a tree has received much attention and is still attracting researchers. The first extremal results on this parameter are due to Sz\'ekely and Wang~\cite {szekely2005subtrees,szekely2007binary}: the structure of binary trees with $n$ leaves that maximise the number of subtrees is given in~\cite{szekely2007binary}, while~\cite{szekely2005subtrees} solves the analogue minimisation problem and also studies $n$-vertex trees that extremise the number of subtrees. For instance, it is known that the $n$-vertex path (resp. $n$-vertex star) has $n(n+1)/2$ (resp. $n-1+2^{n-1}$) subtrees, and these minimise and maximise, respectively, the number of subtrees among all trees of given order. The so-called good binary trees (resp. binary caterpillars) maximise (resp. minimise) the number of subtrees among all binary trees with $n$ leaves. Recently, Chen~\cite{chen2018number} characterised $n$-vertex trees with diameter $d$ that have the maximum number of subtrees, and also solved the minimisation problem in the special case where $d<6$. Results on the number of subtrees with a given degree or segment sequence can be found in~\cite{zhang2013Subtree,zhang2015SubtreeMin}. Paper~\cite{xiao2017trees} mentions that the number of subtrees of a graph was shown in~\cite{ZhaoPhD} to correlate with the reliability of a network with possible vertex/edge failure in the sense that networks with smaller number of subtrees would be less reliable.

\medskip
In this note, we shall determine the structure of all trees that minimise the Wiener index or maximise the number of subtrees, given the eccentric sequence. It will be shown that those extremal trees are caterpillars and coincide in both cases. We shall also provide formulas for the corresponding extreme values of these two invariants. Finally, we mention that the very same tree minimises the edge Wiener index, the vertex-edge Wiener index, the Schultz index, and the Gutman index among all trees with a given eccentric sequence.

\section{Preliminaries and main results}
For graph theoretical terminology not specified here, we refer to~\cite{WestBook}. If $S$ is a subset of vertices (resp. edges) of $G$, then we write $G-S$ to mean the graph obtained from $G$ by deleting all elements of $S$. We simply write $G-l$ instead of $G-\{l\}$. The set of all neighbours of $v\in V(G)$ in $G$ will be denoted by $\mathcal{N}_G(v)$. If $T$ is a caterpillar, then the path that remains after removing its pendant vertices will be called the \emph{backbone} of $T$.

\medskip
The following lemma is well-known; see for instance~\cite{Lesniak1975}. It shall be used without further reference.
\begin{lemma}[\cite{Lesniak1975}]\label{Lemma:calceccent}
Let $u,v$ be two vertices at the maximum distance in a tree $T$. Then we have
$$ \E(w)=\max\{d_T(u,w), d_T(w,v)\}$$ for all $w\in V(T)$.
\end{lemma}
The next result is due to Lesniak~\cite{Lesniak1975} and characterises tree eccentric sequence.
\begin{theorem}\label{MainResult}
For $n>2$, a nondecreasing sequence $S=(a_1,a_2,\ldots,a_n)$ of positive integers is a tree eccentric sequence if and only if 
\begin{enumerate}[i)]
	\item $a_1=a_n/2$ and $a_1\neq a_2$, or $a_1=a_2=(1+a_n)/2$ and $a_2\neq a_3$,
	\item for every integer $a_1<k\leq a_n$, we have $a_j=a_{j+1}=k$ for some $2\leq j\leq n-1$.
\end{enumerate}
\end{theorem}
There are usually many vertices having the same eccentricity in a graph $G$. For this reason, we shall write $(b_1^{(m_1)},b_2^{(m_2)},\ldots, b_l^{(m_l)})$ for the eccentric sequence of $G$, where we mean that $G$ has precisely $l$ distinct eccentricities $b_1 < b_2< \cdots < b_l$ whose multiplicities are $m_1,m_2,\ldots,m_l$. Thus $b_1$ (resp. $b_l$) is the radius (resp. diameter) of $G$ and $|V(G)|=m_1+m_2+\cdots+m_l$. Theorem~1 in~\cite{Lesniak1975} states that $b_{j+1}=b_j +1$ for all $1\leq j \leq l-1$. Therefore, the sequence $(b_1^{(m_1)},b_2^{(m_2)},\ldots, b_l^{(m_l)})$ is completely determined by knowing the values of $b_1,l,m_1,\ldots,m_l$. The parameter $l$, \emph{number of distinct eccentricities}, is a subject of very recent study in~\cite{alizadeh2019eccentric} -- there, it is called the eccentric complexity of $G$. For trees, $m_1=1$ and $b_l=2b_1$, or $m_1=2$ and $b_l=2b_1-1$.

\medskip
Given a tree eccentric sequence $S=(b_1^{(m_1)},b_2^{(m_2)},\ldots, b_l^{(m_l)})$, we shall denote by $\mathcal{T}_S$ the set of all trees whose eccentric sequence is $S$, and by $\mathcal{C}_S$ the set of all caterpillars whose eccentric sequence is $S$. Throughout the paper, we assume that $m_1+m_2+\cdots +m_l>2$.

\begin{definition}
Given an integer $q>0$, we define $\mathbb{T}(t_1,t_2,\ldots,t_r)$ to be the caterpillar constructed from the path $P:v_0,v_1,\ldots,v_{q+1}$ by attaching $t_j$ pendant vertices at $v_j$ for all $1\leq j\leq r$, where $r=q/2$ if $q$ is even, and $r=(q+1)/2$ if $q$ is odd.
\end{definition}

Our main result reads as follows:

\textbf{Main result}: Let $S=(b_1^{(m_1)},b_2^{(m_2)},\ldots, b_l^{(m_l)})$ be a given tree eccentric sequence such that $m_1+m_2+\cdots +m_l>2$. Then $\mathbb{T}(m_l-2,m_{l-1}-2,\ldots,m_2-2)$ minimises the Wiener index and maximises the number of subtrees among all trees whose eccentric sequence is $S$. In each case, $\mathbb{T}(m_l-2,m_{l-1}-2,\ldots,m_2-2)$ is unique with this property.

\medskip
The proof for the minimum Wiener index as well as the corresponding formula is given in Section~\ref{Sec:MinWiener}; that of the maximum number of subtrees as well as the corresponding formula is deferred to Section~\ref{Sec:MaxSubtree}. In the final section, we mention other variants of the Wiener index that the very same tree minimises among all trees with a given eccentric sequence.

\section{Minimum Wiener index}\label{Sec:MinWiener}
In this section, we determine the minimum Wiener index of a tree with a given ecentric sequence $S$, and also characterise all trees attaining the bound.
	
\begin{theorem}\label{MainTheo}
Let $S=(b_1^{(m_1)},b_2^{(m_2)},\ldots, b_l^{(m_l)})$ be a given tree eccentric sequence such that $m_1+m_2+\cdots +m_l>2$. Then we have
$$W(T)> W(\mathbb{T}(m_l-2,m_{l-1}-2,\ldots,m_2-2))$$ for all $T \in \mathcal{T}_S$ such that $T$ is not isomorphic to $\mathbb{T}(m_l-2,m_{l-1}-2,\ldots,m_2-2)$. Moreover, 
\begin{align*}
W\big(&\mathbb{T}(m_l-2,m_{l-1}-2,\ldots,m_2-2)\big)= \binom{b_l+2}{3} ~ + ~ \sum_{j=2}^l (m_j-2)(m_j-3)\\
&+ \sum_{2 \leq i < j \leq l} (m_i-2) (m_j-2) (2+j-i) ~ + ~ \sum_{j=1}^{l-1} \Big( \binom{j}{2} + \binom{b_l+1-j}{2}\Big) (m_{l+1-j}-2)\\
&+(b_l+1)\Big(2-2l+ \sum_{j=2}^l m_j\Big)\,.
\end{align*}   
\end{theorem}

\begin{proof}
Let $T \in \mathcal{T}_S$ be a tree that minimises the Wiener index among all trees whose eccentric sequence is $S$. Let $v_0$ and $v_d$ be two vertices at distance $d:=b_l$ in $T$, and denote by $P=v_0,v_1,\ldots,v_d$ the path from $v_0$ to $v_d$ in $T$.

{\sc Claim~1:} $T$ is a caterpillar. \\
Suppose to the contrary that $T$ is not a caterpillar. Then $P$ contains a vertex $v_j$ which has a neighbour $u$ not on $P$, that is not a pendant vertex. We fix $v_j$ and $u$. We may assume that $j \geq \frac{1}{2}d$ since otherwise, if $j<\frac{1}{2}d$, we just reverse the numbering of the vertices $v_0$ to $v_d$. In order to prove Claim~1, we modify $T$ to obtain a tree $T'$ with the same eccentric sequence but smaller Wiener index, a contradiction to our choice of $T$.

Denote by $U$ the set of vertices that are in a component of $T-u$ not containing $P$. Let $L$ and $R$ be the set of vertices in $V(T)-U$ that are in the component of $T-v_{j}v_{j+1}$ containing $v_j$ and $v_{j+1}$, respectively. Define the tree $T'$ as follows: Delete all edges $u y$ with $y\in \mathcal{N}_T(u)- \{v_j\}$, and add all edges $v_{j+1}y$ with $y\in \mathcal{N}_T(u)- \{v_j\}$. Clearly $V(T)=V(T')$. It is easy to see that for two vertices $x$ and $y$ of $T$ we have $d_{T'}(x,y) \neq d_T(x,y)$ only if  $x\in U$ and $y \in R\cup \{u\}$, or vice versa. For such a pair we have
\[ d_{T'}(x,y) = \left\{ \begin{array}{cc}
d_T(x,y) - 2 & \textrm{if $x\in U$ and $y \in R$,} \\
d_T(x,y) + 2 & \textrm{if $x\in U$ and $y=u$.} 
\end{array} \right. \]
Hence, since $R$ contains more than one vertex, we obtain
\[ W(T') - W(T) = |U|(2 - 2|R|) < 0. \]

\medskip
For the proof that $T$ and $T'$ have the same eccentric sequence, it suffices, by Lemma~\ref{Lemma:calceccent}, to show that all vertices in $U\subset V(T)$ preserve their eccentricities in $T'$. First note that $P$ is also a longest path in $T'$. By $j\geq d-j$, we get
\begin{align*}
\E_T (y)&=\max\{d_T(y,v_0),d_T(y,v_d)\}=d_T(y,u)+1 + j\\
&=d_{T'}(y,v_{j+1})+1+j=\E_{T'} (y)
\end{align*}
for all $y\in U$. Hence $S$ is the eccentric sequence of $T'$ whereas $W(T') < W(T)$. This is a contradiction to our choice of $T$, which proves Claim 1. 

\medskip
As a next step, we bound the Wiener index of a caterpillar $T\in \mathcal{C}_S$ from below. Given the backbone $v_1,\ldots,v_q$ of the caterpillar $T$, fix vertices $v_0 \in \mathcal{N}(v_1)-\{v_2\}$ and $v_{q+1} \in \mathcal{N}(v_q)-\{v_{q-1}\}$. Then $P:=v_0, v_1,\ldots, v_{q+1}$ is a longest path of $T$. For every $i\in \{1,2,\ldots,q\}$, let $C_i$ be the set of pendant vertices adjacent to $v_i$ and not on $P$. For $1 \leq j \leq \lceil \frac{q}{2} \rceil$ define $D_j$ to be the set $C_j \cup C_{q+1-j}$. For $1 \leq j \leq \lceil q/2 \rceil$ the set $D_j$ contains all vertices of eccentricity $q+2-j$, except two vertices that are on $P$. Note that by  $S=(b_1^{(m_1)},b_2^{(m_2)},\ldots, b_l^{(m_l)})$ being the eccentric sequence of $T$, we have
\[\textrm{$b_1=\lceil b_l/2 \rceil$, $b_l=q+1$, and $l=\lfloor (q+1)/2 \rfloor +1$}.\] Hence $|D_j| = m_{l+1-j}-2$. Moreover, the sets $V(P), D_1, D_2,\ldots, D_{\lceil \frac{q}{2} \rceil}$ form a partition of $V(T)$. If $A$ and $B$ are subsets of $V(T)$, then we write $W_T(A)$ for the sum of the distances in $T$ between all unordered pairs of vertices in $A$, and $W_T(A,B)$ for the sum of the distances $d_T(a,b)$, where $a\in A$ and $b \in B$. With this notation, we get 
\[ W(T) = W_T(V(P)) + \sum_{j=1}^{\lceil q/2 \rceil} W_T(D_j) 
~~+ \sum_{1 \leq i < j \leq \lceil q/2 \rceil} W_T(D_i, D_j) 
+ \sum_{j=1}^{\lceil q/2 \rceil} W_T(D_j,V(P)). \]
We consider each of the four terms separately. Clearly 
\[ W_T(V(P)) = W(P) = \binom{q+3}{3}\,. \] The distance between any two vertices in $D_j$ is at least $2$. Therefore, 
\[ W_T(D_j) \geq  2 {|D_j| \choose 2} 
= (m_{l+1-j}-2)(m_{l+1-j}-3)\,.\] Note that equality holds only if all vertices in $D_j$ are adjacent to the same vertex of $P$, i.e., if $C_j$ is empty or $C_{q+1-j}$ is empty. 

To bound $W_T(D_i, D_j)$ for $i<j$, note that for $v \in D_i$ and $w\in D_j$,
we have $d_T(v,w)=d_T(v',w')+2$, where $v'$ and $w'$ is the unique vertex adjacent to
$v$ and $w$, respectively, and $v'\in \{v_i, v_{q+1-i}\}$, $w'\in \{v_j, v_{q+1-j}\}$. 
Since by $i<j$, we have
\begin{align*}
&d_T(v_i,v_j) = d_T(v_{q+1-i}, v_{q+1-j})=j-i\,,\\
&d_T(v_i, v_{q+1-j})= d_T(v_{q+1-i}, v_j) = q+1-i-j \geq j-i\,,
\end{align*}
we derive that
$d_T(v,w) \geq 2+j-i$, with equality only if $j=(q+1)/2$, or
\[\textrm{if $j\neq (q+1)/2$ and  $v \in C_i$ and $w\in C_j$, or if $j\neq (q+1)/2$ and $v\in c_{q+1-i}$ and $w\in C_{q+1-j}$.}
\]
Summation yields
\[ W_T(D_i, D_j) \geq |D_i| \cdot |D_j| (2+j-i) = (m_{l+1-i}-2) (m_{l+1-j}-2) (2+j-i)\,, \]
with equality only if $j=(q+1)/2$, or
\[\textrm{if $j\neq (q+1)/2$ and $C_i=C_j=\emptyset$, or if $j\neq (q+1)/2$ and $C_{q+1-i} = C_{q+1-j} = \emptyset$.}
\]

To evaluate $W_T(D_j,V(P))$, note that for every vertex $v\in C_j$ we have 
$$\sum_{w \in V(P)} d_T(v,w) = \sum_{w \in V(P)} (1+d_T(v_j,w))= (q+2) + W_T(\{v_j\}, V(P)),$$ and similarly for $v \in C_{q+1-j}$, we have 
$$\sum_{w \in V(P)} d_T(v,w) = (q+2) + W_T(\{v_{q+1-j}\}, V(P)).$$ A simple calculation shows that
\begin{align*}
W_T(\{v_j\},V(P)) &= W_T(\{v_{q+1-j}\},V(P))= \frac{1}{2}\big( j(j+1) + (q+1-j)(q+2-j)\big)\\
&=\binom{j+1}{2} + \binom{q+2-j}{2}\,.
\end{align*}
Hence
\begin{align*} 
W_T(D_j,V(P)) &= \Big((q+2)+ \binom{j+1}{2} + \binom{q+2-j}{2}\Big) |D_j| \\
	& = \Big((q+2)+ \binom{j+1}{2} + \binom{q+2-j}{2}\Big) (m_{l+1-j}-2)\,.
\end{align*}                
In total, we have established that
\begin{align}\label{FinalWienerFormula}
\begin{split}
W(T)& = W_T(V(P)) + \sum_{j=1}^{\lceil q/2 \rceil} W_T(D_j) 
~~+ \sum_{1 \leq i < j \leq \lceil q/2 \rceil} W_T(D_i, D_j) 
+ \sum_{j=1}^{\lceil q/2 \rceil} W_T(D_j,V(P))\\
&\geq \binom{q+3}{3} + \sum_{j=1}^{\lceil q/2 \rceil} (m_{l+1-j}-2)(m_{l+1-j}-3)\\
&+ \sum_{1 \leq i < j \leq \lceil q/2 \rceil} (m_{l+1-i}-2) (m_{l+1-j}-2) (2+j-i)\\
& + \sum_{j=1}^{\lceil q/2 \rceil} \Big((q+2)+ \binom{j+1}{2} + \binom{q+2-j}{2}\Big) (m_{l+1-j}-2)\,.
\end{split}
\end{align}              
For $q>1$ and $T\in \mathcal{C}_S$, equality holds in~\eqref{FinalWienerFormula} only if $C_i=C_j=\emptyset$ for all $1\leq i<j \leq \lceil q/2 \rceil$, or if $C_{q+1-i}=C_{q+1-j}=\emptyset$ for all $1\leq i<j \leq \lceil q/2 \rceil$. In other words, equality holds in~\eqref{FinalWienerFormula} for $q>1$ only if $T$ is isomorphic to the caterpillar $\mathbb{T}(m_l-2,m_{l-1}-2,\ldots,m_2-2)$.

On the other hand, for $q=1$, the set $\mathcal{C}_S$ contains only one element which is the tree $\mathbb{T}(m_2-2)$. This completes the proof of the theorem.
\end{proof}

\begin{remark}
It is not hard to see from the proof of Theorem~\ref{MainTheo} that $$\mathbb{T}(\underbrace{0,0,\ldots,0}_{\lfloor d/2 \rfloor -1~0's},n-d-1)$$ is the tree of order $n$ and diameter $d>1$ that has the minimum Wiener index; see also~\cite{Zhou2004}. However, it is a challenging and open problem to determine an exact sharp upper bound on the Wiener index of a graph (or tree) with prescribed order and diameter $>4$. Even in the special case of trees of diameter $5$ or $6$, only asymptotically sharp upper bounds are known; see~\cite{Mukwembi2014} and the references cited therein. Cambie~\cite{Cambie2018} obtained an asymptotically (as $n \to \infty$) sharp upper bound for the Wiener index of a graph with order $n$ and diameter at most $d>2$, and also proved a somewhat analogous result for trees. This suggests that the problem of finding the maximum Wiener index among all trees with a prescribed eccentric sequence can be very difficult.
\end{remark}

\section{Maximum number of subtrees}\label{Sec:MaxSubtree}
We denote the $n$-vertex star by $S_n$. By $\N(T)$ we mean the number of subtrees of a tree $T$. For $u \in V(T)$, we denote by $\N(T)_u$ those subtrees of $T$ that contain $u$.

\medskip
We begin with the following simple lemma, whose proof is left to the reader. 
\begin{lemma}\label{USELem}
Let $t>0$ and $n_1,\ldots,n_t\geq 0$ be fixed integers. Then the function
\begin{align*}
F(x_1,\ldots,x_t)=(2^{x_1}-1)\cdots (2^{x_t}-1) + (2^{n_1-x_1}-1)\cdots (2^{n_t-x_t}-1)
\end{align*}
defined by the inequalities $0\leq x_j\leq n_j$ for all $1\leq j \leq t$, reaches its maximum only at $x_1=\cdots =x_t=0$, or at $x_1=n_1,\ldots x_t=n_t$.
\end{lemma}

\begin{theorem}\label{SubtreeMainTheo}
Let $S=(b_1^{(m_1)},b_2^{(m_2)},\ldots, b_l^{(m_l)})$ be a given tree eccentric sequence such that $m_1+m_2+\cdots +m_l>2$. Then we have
$$\N(T)< \N(\mathbb{T}(m_l-2,m_{l-1}-2,\ldots,m_2-2))$$ for all $T \in \mathcal{T}_S$ such that $T$ is not isomorphic to $\mathbb{T}(m_l-2,m_{l-1}-2,\ldots,m_2-2)$. Furthermore,
\begin{align*}
\N(\mathbb{T}&(m_l-2,m_{l-1}-2,\ldots,m_2-2))= \binom{b_l}{2} -2(l-2) + \sum_{j=2}^l m_j\\
&+ \sum_{p=0}^{b_l-2} \Big( (2^{m_l} -1)\prod_{i=1}^p (2^{m_{l-i}-2} -1) ~+ \sum_{j=2}^{l-3+m_1 -p}  \prod_{i=0}^p (2^{m_{l+1-i-j}-2} -1) \Big)\,.
\end{align*}
\end{theorem}

\begin{proof}
Let $T \in \mathcal{T}_S$ be a tree that maximises the number of subtrees. By mimicking the proof of Theorem~\ref{MainTheo}, we first show that $T$ must be a caterpillar.

{\sc Claim~1:} $T$ is a caterpillar. \\
Suppose to the contrary that $T$ is not a caterpillar. Let $v_0$ and $v_d$ be two vertices at distance $d:=b_l$ in $T$, and denote by $P=v_0,v_1,\ldots,v_d$ the path from $v_0$ to $v_d$ in $T$. Then $P$ contains a vertex $v_j$ which has a neighbour $u$ not on $P$, that is not a pendant vertex. We fix $v_j$ and $u$. Clearly, we may assume that $j \geq \frac{1}{2}d$. In order to prove Claim~1, we modify $T$ to obtain a tree $T'$ with the same eccentric sequence but greater number of subtrees, a contradiction to our choice of $T$.

Denote by $U$ the set of vertices that are in a component of $T-u$ not containing $P$. Let $L$ and $R$ be the components of $(T-U)-v_{j}v_{j+1}$ containing $v_j$ and $v_{j+1}$, respectively. Define the tree $T'$ as follows: Delete all edges $u y$ with $y\in \mathcal{N}_T(u)- \{v_j\}$, and add all edges $v_{j+1}y$ with $y\in \mathcal{N}_T(u)- \{v_j\}$. It was shown in the proof of Theorem~\ref{MainTheo} that $S$ is the eccentric sequence of $T'$. Since $T-U$ is isomorphic to $T'-U$, the number of subtrees of $T-U$ equals the number of subtrees of $T'-U$. Also the number of subtrees with vertex set contained in $U$ is the same for $T$ and $T'$. Thus in order to prove that $\N(T')>\N(T)$, it suffices to compare the number of subtrees of $T$ that contain both a vertex in $U$ and a vertex of $T-U$ with the number of subtrees of $T'$ that contain both a vertex in $U$ and a vertex of $T'-U$. It is easy to see that the following hold:
\begin{align*}
\N(T-U)_u&=1 + \N(L-u)_{v_j}(1+\N(R)_{v_{j+1}})\,,\\ \N(T'-U)_{v_{j+1}}&=\N(R)_{v_{j+1}} (1+ 2\N(L-u)_{v_j})\,.
\end{align*}
The difference between these two quantities gives
\begin{align*}
\N(T'-U)_{v_{j+1}}& - \N(T-U)_u =(\N(L-u)_{v_j}+1) (\N(R)_{v_{j+1}} -1)>0\,,
\end{align*}
where the strict inequality is due to the fact that $R$ contains more than one element. Let $B$ be the tree induced by $U\cup \{u\}$ in $T$, and $B'$ the tree induced by $U\cup \{v_{j+1}\}$ in $T'$. Every subtree of $T$ that involves both a vertex in $U$ and a vertex of $T-U$ can be obtained by merging at $u$ a $u$-containing subtree of $B$ and a $u$-containing subtree of $T-U$. Similarly, every subtree of $T'$ that involves both a vertex in $U$ and a vertex of $T'-U$ can be obtained by merging at $v_{j+1}$ a $v_{j+1}$-containing subtree of $B'$ and a $v_{j+1}$-containing subtree of  $T'-U$. Thus by $\N(B)_u=\N(B')_{v_{j+1}}$, we get
\begin{align*}
\N(T')-\N(T)&=\N(B')_{v_{j+1}}\cdot \N(T'-U)_{v_{j+1}} - \N(B)_u \cdot \N(T-U)_u\\
&=\N(B)_u (\N(T'-U)_{v_{j+1}} - \N(T-U)_u)>0\,,
\end{align*}
which is a contradiction to the maximality of $T$. Hence, every vertex of $T$ not lying on $P$ must be adjacent to some vertex of $P$, that is $T$ must be a caterpillar.

\medskip
We now derive the structure of the specific caterpillar whose eccentric sequence is $S$ and that has the maximum number of subtrees. Note that $\N(S_n)=2^{n-1} +n-1$ for all $n$. We shall frequently make use of Lemma~\ref{USELem}. Fix a caterpillar $T\in \mathcal{C}_S$ and let $P=v_1,\ldots,v_q$ be the backbone of $T$. Fix vertices $v_0 \in \mathcal{N}(v_1)-\{v_2\}$ and $v_{q+1} \in \mathcal{N}(v_q)-\{v_{q-1}\}$. Then $v_0, v_1,\ldots, v_{q+1}$ is a longest path of $T$. For every $i\in \{1,2,\ldots,q\}$ let $C_i$ be the set of pendant vertices adjacent to $v_i$. For $1 \leq j \leq \lceil \frac{q}{2} \rceil$ define $D_j$ to be the set $C_j \cup C_{q+1-j}$. For $1 \leq j \leq \lceil q/2 \rceil$ the set $D_j$ contains all vertices of eccentricity $q+2-j$, except two vertices that are on $P$ in the case where $j \notin \{1,q\}$. Thus by  $S=(b_1^{(m_1)},b_2^{(m_2)},\ldots, b_l^{(m_l)})$ being the eccentric sequence of $T$, we have
\[\textrm{ $|D_1| = m_l$ and $|D_j| = m_{l+1-j}-2$ for all  $j \neq 1$.}\] Moreover, the sets $V(P), D_1, D_2,\ldots, D_{\lceil \frac{q}{2} \rceil}$ form a partition of $V(T)$. Note that this partition is slightly different from the partition of $V(T)$ given in the proof of Theorem~\ref{MainTheo}.  

If $A_1,A_2,\ldots,A_k$ are subsets of $V(T)$, then we write $\N_T(A_1)$ for the sum of the number of subtrees induced in $T$ by all (non-empty) subsets of $A_1$, and $\N_T(A_1,A_2,\ldots,A_k)$ for the sum of the number of subtrees induced in $T$ by all subsets $S=S_1\cup S_2\cup \cdots \cup S_k$ where for every $j\in\{1,2,\ldots,k\}$, $S_j$ is a non-empty subset of $A_j$. With this notation, we get 
\begin{align*}
\N(T)&=\N_T(V(P))+\sum_{j=1}^{\lceil q/2 \rceil}\N_T(D_j) ~+\sum_{k=2}^{\lceil q/2 \rceil} \sum_{1\leq i_1<\cdots < i_k \leq \lceil q/2 \rceil} \N_T(D_{i_1},\ldots,D_{i_k}) \\
& \quad +\sum_{k=1}^{\lceil q/2 \rceil} \sum_{1\leq i_1<\cdots < i_k \leq \lceil q/2 \rceil} \N_T(D_{i_1},\ldots,D_{i_k},V(P)) \,.
\end{align*}
Clearly, 
\[ \textrm{$\N_T(V(P)) = \N(P) = \binom{q+1}{2}$ ~~ and ~~ $\N_T(D_j) = |D_j|$.}\] For a subtree $B$ of $T$ to contain both a vertex in $C_i$ and a vertex in $C_j$ for $i<j$, it is necessary for $B$ to contain the entire path $v_i,v_{i+1},\ldots,v_j$. Thus \[\textrm{$\N_T(D_{i_1},\ldots,D_{i_k})=0$ for all $i_1<\cdots < i_k$ such that $k>1$.}\] Moreover, only those subsets of $V(P)$ that induce a path in $T$ contribute to the quantity $\N_T(D_{i_1},\ldots,D_{i_k},V(P))$. In addition, for given integers $k\geq 1,~ p\geq 0$ and a set $\{v_j,v_{j+1},\ldots, v_{j+p}\}\subseteq V(P)$, only those indices $i_1,\ldots,i_k$ such that 
\[\textrm{$i_t \in \{j,j+1,\ldots,j+p\}$ or $i_{q+1-t} \in \{j,j+1,\ldots,j+p\}$ for all $1\leq t \leq k$}
\]
contribute to $\N_T(D_{i_1},\ldots,D_{i_k},\{v_j,v_{j+1},\ldots, v_{j+p}\})$. Therefore, we have
\begin{align*}
\sum_{k=1}^{\lceil q/2 \rceil}  \sum_{1\leq i_1<\cdots < i_k \leq \lceil q/2 \rceil} \N_T(D_{i_1},\ldots,D_{i_k},\{v_j,v_{j+1},\ldots, v_{j+p}\}^*)= \prod_{i=0}^p (2^{|C_{j+i}|} -1)
\end{align*}
for all $0\leq p \leq q-1$ and all $1\leq j \leq q -p$, where by $\N_T(D_{i_1},\ldots,D_{i_k},\{v_j,v_{j+1},\ldots, v_{j+p}\}^*)$ we mean the number of those subtrees that involve entirely $v_j,v_{j+1},\ldots, v_{j+p}$. It follows that
\begin{align*}
\sum_{k=1}^{\lceil q/2 \rceil} &\sum_{1\leq i_1<\cdots < i_k \leq \lceil q/2 \rceil} \N_T(D_{i_1},\ldots,D_{i_k},V(P)) = \sum_{p=0}^{q-1}\sum_{j=1}^{q-p}  \prod_{i=0}^p (2^{|C_{j+i}|} -1)\\
&=\sum_{p=0}^{q-1}\sum_{j=1}^{\lfloor q/2 \rfloor -p} \Big( \prod_{i=j}^{p+j} (2^{|C_i|} -1) ~ +  \prod_{i=q+1-p-j}^{q+1-j} (2^{|C_i|} -1) \Big)\,.
\end{align*}
On the other hand, for every $p \in \{0,1,\ldots,q-1\}$ and $ j \in \{1,2,\ldots,\lfloor q/2 \rfloor -p\}$, we have
\begin{align*}
\prod_{i=j}^{p+j} (2^{|C_i|} -1) ~ +  \prod_{i=q+1-p-j}^{q+1-j} (2^{|C_i|} -1)  \leq (2^{|D_j|} -1)(2^{|D_{j+1}|} -1)\cdots (2^{|D_{j+p}|} -1)\,,
\end{align*}
with equality only if $C_j=C_{j+1}=\cdots=C_{j+p}=\emptyset$, or if $C_{q+1-j-p}=C_{q+1-j-p+1}=\cdots=C_{q+1-j}=\emptyset$. In total, we have established that
\begin{align*}
\N(T)&=\N_T(V(P))+\sum_{j=1}^{\lceil q/2 \rceil}\N_T(D_j) ~ + ~\sum_{k=1}^{\lceil q/2 \rceil} ~ \sum_{1\leq i_1<\cdots < i_k \leq \lceil q/2 \rceil} \N_T(D_{i_1},\ldots,D_{i_k},V(P))\\
&\leq \binom{q+1}{2} + \sum_{j=1}^{\lceil q/2 \rceil} |D_j| + \sum_{p=0}^{q-1}\sum_{j=1}^{\lfloor q/2 \rfloor -p}  (2^{|D_j|} -1)(2^{|D_{j+1}|} -1)\cdots (2^{|D_{j+p}|} -1)\,,
\end{align*}
with equality for $T\in \mathcal{C}_S$ only if $C_1=C_2=\cdots=C_{\lfloor q/2 \rfloor}=\emptyset$, or  $C_{q}=C_{q-1}=\cdots=C_{1+\lceil q/2 \rceil}=\emptyset$. In other words, equality holds for $T\in \mathcal{C}_S$ only if $T$ is isomorphic to the caterpillar $\mathbb{T}(m_l-2,m_{l-1}-2,\ldots,m_2-2)$. Recall that $q+1=b_l,~b_l-b_1=\lfloor b_l/2 \rfloor=\lceil q/2 \rceil$ and that
\[\textrm{ $|D_1| = m_l$ and $|D_j| = m_{l+1-j}-2$ for all  $j \neq 1$.}\] In particular, we obtain
\begin{align*}
\N(\mathbb{T}(m_l-2,&m_{l-1}-2,\ldots,m_2-2))= \binom{q+1}{2} + \sum_{j=1}^{\lceil q/2 \rceil} |D_j| + \sum_{p=0}^{q-1}\sum_{j=1}^{\lfloor q/2 \rfloor -p}  \prod_{i=0}^p (2^{|D_{j+i}|} -1)\\
&=\binom{b_l}{2} + m_l+ \sum_{j=2}^{l-1} (m_j-2) \\
&\quad \quad + \sum_{p=0}^{b_l-2} \Big( (2^{m_l} -1)\prod_{i=1}^p (2^{m_{l-i}-2} -1) ~+ \sum_{j=2}^{l-3+m_1 -p}  \prod_{i=0}^p (2^{m_{l+1-i-j}-2} -1) \Big)
\end{align*}
since $\lfloor q/2 \rfloor =l-3+m_1$. This completes the proof of the theorem.
\end{proof}

\begin{remark}
It is not hard to see from the proof of Theorem~\ref{SubtreeMainTheo} that $$\mathbb{T}(\underbrace{0,0,\ldots,0}_{\lfloor d/2 \rfloor -1~0's},n-d-1)$$ is the tree of order $n$ and diameter $d>1$ that has the maximum number of subtrees; see also~\cite{chen2018number}. However, it is an open problem to determine an exact sharp lower bound on the number of subtrees with order $n$ and diameter $d>5$. This suggests that the problem of finding the minimum number of subtrees among all trees with a prescribed eccentric sequence can be very difficult.
\end{remark}

\section{Concluding remarks}
Let $G$ be a connected graph whose vertex and edge sets are $V(G)$ and $E(G)$, respectively. For $u\in V(G)$ and $e=vw \in E(G)$, the distance between vertex $u$ and edge $e$ is $\min\{d(u,v),d(u,w)\}$. For $f \in E(G)$, the distance (as defined in~\cite{khalifeh2009some}) between edges $e$ and $f$ is $\min \{d(v,f),d(w,f)\}$. There are some variants of the Wiener index of a graph, which include the \emph{edge Wiener index} $W_e$, the \emph{vertex-edge Wiener index} $W_{ve}$, the \emph{Schultz index} (also known as the \emph{degree distance}) $W_{+}$, and the Gutman index $W_{-}$. They are defined as
\begin{align*}
W_e(G)&=\sum_{\{e,f\}\subseteq E(G)} d(e,f)\,, \quad W_{ve}(G)=\frac{1}{2}\sum_{\substack{v\in V(G)\\ e\in E(G)}} d(v,e)\,, \\ W_{+}(G)&=\sum_{\{u,v\}\subseteq V(G)} d(u,v) (\deg(u) + \deg(v))\,, ~ W_{-}(G)=\sum_{\{u,v\}\subseteq V(G)} d(u,v) \deg(u) \deg(v)\,,
\end{align*}
where $\deg(u)$ denotes the degree of $u$ in $G$. In~\cite{gutmanschultz1994,khalifeh2009some,klein1992molecular} it was shown that for $n$-vertex trees $T$, all the above invariants are closely related to the Wiener index, namely that  
\begin{align*}
W_e(T)&=W(T)-(n-1)^2\,, \quad  W_{ve}(T)=W(T)-n(n-1)/2\,, \\
W_{+}(T)&=4W(T) -n(n-1)\,, \quad W_{-}(T)=4W(T) -(n-1)(2n-1)\,.
\end{align*}

There is yet another measure of distance between two edges. In~\cite{DankelmannEdgeWiener} the distance $d'(e,f)$ between edges $e$ and $f$ of $G$ is defined to be the distance between the corresponding vertices in the line graph of $G$. It is easy to see that
$d'(e,f)=d(e,f)+1$ and that for $n$-vertex trees $T$, 
\begin{align*}
W'_e(T):=\sum_{\{e,f\}\subseteq E(G)} d'(e,f)=W'_e(T) +n-1\,.
\end{align*}

\begin{corollary}
Let $\mathcal{I}(.) \in \{W'_e(.), W_e(.),W_{ve}(.),W_{+}(.),W_{-}(.)\}$. Then $\mathbb{T}(m_l-2,m_{l-1}-2,\ldots,m_2-2)$ is the unique tree with eccentric sequence $S=(b_1^{(m_1)},b_2^{(m_2)},\ldots, b_l^{(m_l)})$ that minimises $\mathcal{I}(.)$.
\end{corollary}

On the other hand, there are other Wiener-type indices which are in no relation with the Wiener index. Some of them were shown to correlate better with various physico-chemical properties of certain molecules and molecular structures than the classical Wiener index~\cite{gutman2003note,randic2001interpretation}. Two such variants include
\[\textrm{$HW(G)=\sum_{{u,v}\subseteq V(G)}\binom{1+d(u,v)}{2}$ ~~ and ~~ $
W(G;\lambda)=\sum_{{u,v}\subseteq V(G)} d(u,v)^{\lambda}$,}\]
where $\lambda \neq 0$ is any given real number. It seems (experimentally) that given the eccentric sequence, the tree $\mathbb{T}(m_l-2,m_{l-1}-2,\ldots,m_2-2)$ minimises both $HW(G)$ and $W(G;\lambda)$ for every $\lambda \geq 1$. The authors are continuing this investigation.


\begin{thebibliography}{11} 
	
	\bibitem{alizadeh2019eccentric}
	Y.~Alizadeh, T.~Do{\v{s}}li{\'c}, and K.~Xu.
	\newblock On the eccentric complexity of graphs.
	\newblock {\em Bull. Malaysian Math. Sci. Soc.}, 42(4):1607--1623, 2019.

	\bibitem{PrevPaper}
	E.~O.~D.~Andriantiana, S.~Wagner and H.~Wang.
	\newblock Maximum Wiener index of trees with given segment sequence.
	\newblock {\em MATCH Commun. Math. Comput. Chem.}, 75(1):91--104, 2016.
	
	\bibitem{Behzad1976}
    M.~Behzad and J.~E.~Simpson.
	\newblock Eccentric sequences and eccentric sets in graphs.
	\newblock {\em Discrete Math.}, 16(3):187--193, 1976.
	
		
	\bibitem{Buckley2002}
	F.~Buckley and F.~Harary.
	\newblock Unsolved problems on distance in graphs.
	\newblock {\em Electr. Notes Discr. Math.}, 11:89--97, 2002.
	
	\bibitem{Cambie2018}
	S.~Cambie.
	\newblock The asymptotic resolution of a problem of Plesn\'ik.
	\newblock Preprint arXiv:1811.08334, 2018.	
	
	\bibitem{Cela}
	E.~\c{C}ela, N.~S.~Schmuck, S.~Wimer, G.~J.~Woeginger.
	\newblock The Wiener maximum quadratic assignment problem.
	\newblock {\em Discr. Optimizat.}, 8(3):411-416, 2011.
	
	\bibitem{chen2018number}
	Z.~Chen.
	\newblock The number of subtrees of trees with given diameter.
	\newblock {\em Discr. Math., Algor.  Applicat.}, 10(01):1850001, 2018.
	
	\bibitem{DankelmannMinDegree}
	P.~Dankelmann.
	\newblock Average distance, minimum degree, and size.
	\newblock {\em Utilitas Math.}, 69:233-243, 2006.
	
	\bibitem{DankelmannEtringer}
	P.~Dankelmann and R.~Entringer.
	\newblock Average distance, minimum degree, and spanning trees.
	\newblock {\em J. Graph Theory}, 33(1):1--13, 2000.
	
	\bibitem{DankelmannMukwembiEdge}
	P.~Dankelmann, S.~Mukwembi, and H.~C.~Swart.
	\newblock Average distance and edge-connectivity.
	\newblock {\em SIAM J. Discr. Math.}, 22(1):92--101, 2008.
	
	\bibitem{DankelmannMukwembi}
	P.~Dankelmann, S.~Mukwembi, and H.~C.~Swart.
	\newblock Average distance and vertex-connectivity.
	\newblock {\em J. Graph Theory}, 62(2):157--177, 2009.
	
	\bibitem{DankelmannEdgeWiener}
	P.~Dankelmann, I.~Gutman, S.~Mukwembi, and H.~C.~Swart.
	\newblock The edge-Wiener index of a graph.
	\newblock {\em Discr. Math.}, 309(10):3452--3457, 2009.
	
	
	\bibitem{Dankelmann2014}
	P.~Dankelmann, D.~J.~Erwin, W.~Goddard, S.~Mukwembi, and H.~C.~Swart.
	\newblock A characterisation of eccentric sequences of maximal outerplanar graphs.
	\newblock {\em Australas. J. Combin.}, 58(3):376--391, 2014.

\bibitem{DobryninAll}
A.~A.~Dobrynin, R.~Entringer, and I.~Gutman.
\newblock Wiener index of trees: theory and applications.
\newblock {\em  Acta Appl. Math.}, 66(3):211--249, 2011.

\bibitem{gutmanschultz1994}
I.~Gutman.
\newblock Selected properties of the Schultz molecular topological.
\newblock {\em J. Chem. Inf. Comput.}, 34(5):1087--1089, 1994.

\bibitem{gutman2003note}
I.~Gutman, B.~Furtula, and J.~ Beli{\'c}.
\newblock Note of the hyper-Wiener index.
\newblock {\em J. Serbian Chem. Soci.}, 68(12):943--948, 2003.

	\bibitem{Haviar1997}
	A.~Haviar, P.~Hrn\u{c}iar, and G.~Monoszov\'{a}.
	\newblock Minimal eccentric sequences with least eccentricity three.
	\newblock {\em Acta Univ. Mathaei Belii Nat. Sci. Ser. Math.}, 5:27--50, 1997.
	
	\bibitem{Monoszova2005}
	P.~Hrn\u{c}iar, and G.~Monoszov\'{a}.
	\newblock Minimal eccentric sequences with two values.
	\newblock {\em Acta Univ. Mathaei Belii Nat. Sci. Ser. Math.}, 12:43--65, 2005.
		
%
	\bibitem{khalifeh2009some}
	M.~H.~Khalifeh, H.~Yousefi-Azari, A.~R.~Ashrafi, and S.~G.~Wagner.
	\newblock Some new results on distance-based graph invariants.
	\newblock {\em Euro. J. Combin.}, 30(5):1149--1163, 2009.
		
	\bibitem{klein1992molecular}
	D.~J.~Klein, Z.~Mihali\'{c}, D.~Plav\v si\'{c}, and N.~Trinajsti\'{c}.
	\newblock Molecular topological index: A relation with the Wiener index.
	\newblock {\em J. Chem. Inform. Comput. Sci.}, 32(4):304--305, 1992.
	
	\bibitem{Lesniak1975}
	L.~Lesniak.
	\newblock Eccentric sequences in graphs.
	\newblock {\em Period. Math. Hungar.}, 6(4):287--293, 1975.
	
	\bibitem{LinSong}
	H.~Lin and M.~Song.
	\newblock On segment sequences and the Wiener index of trees.
	\newblock {\em MATCH Commun. Math. Comput. Chem.}, 75(1):81–--89, 2016.

\bibitem{MukwembiIndepend}
S.~Mukwembi.
\newblock Average distance, independence number, and spanning trees.
\newblock {\em J. Graph Theory}, 76(3):194--199, 2014.

\bibitem{Mukwembi2014}
S.~Mukwembi and T.~Vetr\`ik.
\newblock Wiener index of trees of given order and diameter at most 6.
\newblock {\em Bull. Aust. Math. Soc.}, 89(3):379--396, 2014.

	\bibitem{Nandakumar1986}
	R.~Nandakumar.
	\newblock On some eccentric properties of graphs.
	\newblock PhD Thesis, India Institute of Technology, India, 1986.
	
		\bibitem{plesnik1984sum}
		J.~Plesn{\'\i}k.
		\newblock On the sum of all distances in a graph or digraph.
		\newblock {\em J. Graph Theory}, 8(1):1--21, 1984.

\bibitem{randic2001interpretation}
M.~Randi{\'c} and J.~Zupan.
\newblock On interpretation of well-known topological indices.
\newblock {\em J. Chem. Inform. Comput. Sci.}, 41(3):550--560, 2001.

	\bibitem{Shi1993}
	R.~Shi.
	\newblock The average distance of trees.
	\newblock {\em Sys. Sci. Math. Sci.}, 6(1):18--24, 1993.

	\bibitem{Skurnik1999}
	R.~Skurnik.
	\newblock Eccentricity sequences of trees and their caterpillars cousins.
	\newblock {\em Graph Theory Notes New York}, (37):20--24, 1999.

\bibitem{szekely2005subtrees}
L.~Sz{\'e}kely and H.~Wang.
\newblock On subtrees of trees.
\newblock {\em Adv. Appl. Math.}, 34(1):138-155, 2005.

\bibitem{szekely2007binary}
L.~Sz{\'e}kely and H.~Wang.
\newblock Binary trees with the largest number of subtrees.
\newblock {\em Discr. Appl. Math.}, 155(3):374--385, 2007.

	\bibitem{wagner2007correlation}
	S.~Wagner.
	\newblock Correlation of graph-theoretical indices.
	\newblock {\em SIAM J. Discrete Math.}, 21(1):33--46, 2007.
	
	\bibitem{WangP2009}
	H.~Wang.
	\newblock The extremal values of the Wiener index of a tree with given degree sequence.
	\newblock {\em Discrete Appl. Math.}, 156(14):2647--2654, 2009.

	\bibitem{wang2008extremal}
	H.~Wang.
	\newblock Corrigendum: The extremal values of the Wiener index of a tree with given degree sequence.
	\newblock {\em Discr. Appl. Math.}, 157(18):3754, 2009.

\bibitem{WestBook}
D.~B.~West et al.
\newblock Introduction to graph theory.
\newblock Prentice hall, Upper Saddle River, Vol. 2, 2001.

	\bibitem{Wiener1947}
	H.~Wiener. 
	\newblock Structural determination of paraffin boiling points. 
	\newblock {\em J. Amer. Chem. Soc.}, 69(1):17--20, 1947.
	
	\bibitem{xiao2017trees}
	Y.~Xiao, H.~Zhao, Z.~Liu, and Y.~Mao.
	\newblock Trees with large numbers of subtrees.
	\newblock {\em International J. Comput. Math.}, 94(2):372--385, 2017.
		
	\bibitem{xu2014survey}
	K.~Xu, M.~Liu, K.~C.~Das, I.~Gutman and B.~Furtula.
	\newblock A survey on graphs extremal with respect to distance-based topological indices.
	\newblock {\em MATCH Commun. Math. Comput. Chem}, 71(3):461--508, 2014.
	
	\bibitem{zhang2008wiener}
	X.-D.~Zhang, Q.-Y.~Xiang, L.-Q.~Xu and R.-Y.~Pan.
	\newblock The Wiener index of trees with given degree sequences.
	\newblock {\em MATCH Commun. Math. Comput. Chem}, 60(2):623--644, 2008.

\bibitem{zhang2013Subtree}
X.-M.~Zhang, X.-D.~Zhang, D.~Gray, and H.~Wang.
\newblock The number of subtrees of trees with given degree sequence.
\newblock {\em J. Graph Theory}, 73(3):280--295, 2013.

\bibitem{zhang2015SubtreeMin}
X.-M. Zhang and X.-D. Zhang.
\newblock The minimal number of subtrees with a given degree sequence.
\newblock {\em Graphs and Combin.}, 31(1):309--318, 2015.

	\bibitem{ZhaoPhD}
	H.~X.~Zhao.
	\newblock Research on subgraph polynomials and reliability of networks.
	\newblock 	Ph.D. thesis, School of Computer, Northwestern Polytechnical University, Xi’an, China, 2004.

	\bibitem{Zhou2004}
	T.~Zhou, J.~Xu, and J.~Liu, 
	\newblock On diameter and average distance of graphs.
	\newblock {\em OR Trans.}, 8(4):1--6. 2004.
	
	
	

	
\end{thebibliography}
\end{document}